\documentclass{comjnl}
\usepackage{amsfonts}
\usepackage{cite}
\usepackage{autobreak}
\usepackage{stfloats}
\usepackage{float} 
\usepackage{subfigure} 
\usepackage{algorithmic}
\usepackage{algorithm}
\usepackage{graphicx}
\usepackage{mathrsfs}
\usepackage{enumitem}  
\usepackage{xcolor}
\usepackage{makecell} 

\usepackage{float}
\usepackage{booktabs}
\newtheorem{claim}{Claim} 

\begin{document}

\title[Paired 2-disjoint path covers of BCube under the partitioned edge fault model]{Paired 2-disjoint path covers of BCube under the partitioned edge 
fault model}

\author{%
  Wen-Jing Zhang{\textsuperscript{1}}, 
  Qing-Qiong Cai{\textsuperscript{1,2}}\thanks{},
  Jou-Ming Chang{\textsuperscript{3}}
  }

\affiliation{%
  1.College of Computer Science, Nankai University, Tianjin, 300350, China\\
  2.Academy for Advanced Interdisciplinary Studies, Nankai University, Tianjin, 300350, China\\
  3. Institute of Information and Decision Sciences, National Taipei University of Business, Taipei 10051, Taiwan.
}


\email{caiqingqiong@nankai.edu.cn}

\shortauthors{Chang, Zhang, Cai}







\shortauthors{Author}


\begin{abstract}
BCube, as a popular server-centric data center network (DCN), offers significant advantages in low latency, load balancing, and high bandwidth. The many-to-many paired $m$-disjoint path cover ($m$-DPC), a generalization of Hamiltonian paths, enhances message transmission efficiency by constructing disjoint paths that connect $m$ source-destination pairs while covering all the nodes. However, with the continuous expansion of DCNs, link and service failures have grown increasingly common, 
necessitating robust fault-tolerant algorithms to guarantee reliable communication.
This paper mainly investigates the fault-tolerant paired 2-DPC embedding in BCube. We prove that under the partitioned edge fault (PEF) model, BCube retains a paired 2-DPC even when exponentially many edge failures occur. 

\end{abstract}

\maketitle

\section{Introduction}
With the rapid advancement of cloud computing, data center networks (DCNs) have emerged as the critical infrastructure underpinning modern information technology. To accommodate the growing demands for large-scale computing and storage, data centers continue to scale up, integrating thousands of servers and switches. At the same time, the topological structures of DCNs are becoming increasingly complex, bringing to the fore significant challenges in scalability, reliability, and energy efficiency. DCN architectures can be broadly categorized into two types: \textit{server-centric} and \textit{switch-centric}. For analytical convenience, we model these architectures as logical graphs, where servers are represented as nodes, switches are viewed as transparent that connect nodes corresponding to servers to form clusters, and communication links are depicted as edges. Representative server-centric DCNs include BCube \cite{2009BCube}, HSDC \cite{2019HSDC}, and DCell \cite{2008DCell}, while notable switch-centric architectures include DPCell\cite{2021DPCell} and Fat-tree\cite{2008A}.

BCube, proposed by Guo et al. \cite{2009BCube}, is a server-centric network architecture designed for efficient and robust DCNs. It features a recursively constructed topology that relies on low-cost mini-switches to interconnect servers, contributing to reduced manufacturing costs. BCube has been adopted in practical data center solutions such as Sun's Modular Datacenter (MD) \cite{SunMD} and HP's Performance-Optimized Datacenter (POD) \cite{HPE_POD}. In terms of transmission performance, BCube exhibits a low network diameter, which helps decrease latency in both one-to-many and many-to-many communication patterns. Compared to alternative structures like DCell \cite{2008DCell} and Fat-tree\cite{1994Fat}, BCube offers additional benefits such as improved load balancing and higher bandwidth. These characteristics make BCube an ideal model for constructing large-scale data center networks.

Many parallel and distributed computing systems adopt path or cycle-based topologies as their underlying infrastructures, which effectively support data transmission and algorithm design \cite{2021High}. Among these, the \emph{Hamiltonian path}—a path that traverses every node exactly once—holds particular significance in networking.  Hamiltonian paths ensure that each node is visited precisely once, making them especially suitable for multicast routing algorithms that help prevent deadlocks and congestion \cite{2015The}. A graph is referred to as \emph{Hamiltonian-connected} if there exists a Hamiltonian path between every pair of distinct nodes.

The concept of \emph{many-to-many m-Disjoint Path Cover ($m$-DPC)} generalizes the Hamiltonian path. It consists of 
m disjoint paths that connect different source-destination pairs while collectively covering all nodes in the network. Note that when 
$m = 1$, the $m$-DPC reduces to a Hamiltonian path. In distributed DCNs, $m$-DPC can significantly improve message transmission efficiency for tasks such as broadcasting and information gathering. Moreover, especially when $m\geq2$, it finds broad applications in areas such as software testing, topology control in sensor networks, and code optimization \cite{2013Paired}.
There has been extensive research on Hamiltonian embedding and $m$-DPC embedding in various network topologies \cite{2016An,2018Conditional,2015Hamiltonian,2022Unpaired,2006Many,2009Many}.

During the operation of DCNs, failures may occur in switches, servers, or communication links. Link failures occur most frequently \cite{2011Understanding}. Compared to servers, switches are more prone to failures \cite{2015Jupiter}. Switch failures can be converted into link failures, and once a switch fails, it can cause a large number of link failures. Link failures or switch failures may result in substantial communication overhead and severely affect the normal communication of the entire network\cite{2025Link}. Therefore, it is crucial to evaluate the edge-fault-tolerance and design corresponding fault-tolerant algorithms. This article focuses specifically on the fault-tolerance of network architectures under edge failures.

For a graph $G$ and a given set of faulty edges $F$, $G$ is said to be \emph{$F$-fault-tolerant Hamiltonian-connected} if there exists a Hamiltonian path between every pair of distinct nodes in $G - F$. Similarly, $G$ is \emph{$F$-fault paired 2-disjoint path coverable} if for any set of source nodes ${s_1, s_2}$ and destination nodes ${t_1, t_2}$, there exists a paired 2-disjoint path cover (2-DPC) in $G - F$.
There is a series of work on fault-tolerant Hamiltonian connectivity and 2-DPC. Wang et al. \cite{2020Fault} showed that BCube($n$,$k$) is $((n-1)(k+1)-2)$-fault-tolerant Hamiltonian and $((n-1)(k+1)-3)$-fault-tolerant Hamiltonian-connected. Wang et al. \cite{2015Hamiltonian} established that $DCell_k$ is $(n+k-3)$-fault Hamiltonian and $(n+k-4)$-fault Hamiltonian-connected. Lu \cite{2019Paired} proved that the balanced hypercube $BH_n$ admits a paired 2-DPC even with up to $2n - 3$ faulty edges.  Under the assumption that $k_i \geq 3$ for all $i = 1, 2, \dots, n$ and at most one $k_i$ is even, it was shown that the torus network $T(k_1, k_2, \dots, k_n)$ still has a paired 2-DPC in the presence of up to $2n - 3$ faulty edges \cite{LI20171}.

Notably, these fault tolerance bounds grow linearly with the size or dimension of the graph, indicating limited fault tolerance capability. These studies are based on the \emph{random fault model}, where faulty edges occur randomly without any restrictions.
This model allows the possibility that all edges incident to a node fail simultaneously—a scenario that is statistically rare in practice \cite{0Link}. In fact, probabilistic analysis confirms that the likelihood of all edges around a node failing is extremely low.

Under a \emph{conditional fault model} that requires at least two fault-free edges to be incident to each node, we can derive from the results of Cheng et al. \cite{Cheng2013Conditional} the following conclusions: for $n \in \{7,9\}$ and $k \geq 0$, BCube($n$,$k$) is Hamiltonian even with up to $2(k+1)(n-1)-7$ faulty edges; for $n \geq 6$ and $n \notin \{7,9\}$, this upper bound on the number of faulty edges is improved to $2(k+1)(n-1)-6$. These results demonstrate that imposing realistic constraints on the fault model can improve achievable fault-tolerance. However, even under such conditions, exponential-level fault-tolerance remains out of reach.




By observing various faulty characteristics of different dimensions of $k$-ary $n$-cube $Q_{n, k}$ in practical applications, Zhuang et al. \cite{2023An} 
introduced the \emph{partitioned edge fault model} (PEF model), which groups edges by dimension. 
Under this model, they established that $Q_{n,k}$ can preserve Hamiltonian connectivity 
when the total number of faulty edges is at most $(k^n - k^2)/(k - 1) - 2n + 5$, subject to the per-dimension constraints $e_i \leq k^i - 2$ for $2 \leq i \leq k$, $e_1 \leq 1$, and $e_0 = 0$. Here $\{e_0,e_1,...,e_k\}$ is the set of the number of faulty edges in each dimension with $e_0\leq e_1\leq ...\leq e_k$.
In subsequent work \cite{2023Embedding, 2024Paired}, they studied the fault-tolerant paired 2-DPC and (hyper-) Hamiltonian laceability of $Q_{n,k}$ under PEF model.
Using this partitioning method, the balanced hypercube \(BH_n\) is shown to be \((2^{n - 1})\)-partition-edge fault-tolerant Hamiltonian for $n \geq 2$ \cite{2023Novel}, and \((2^{n - 1} - 1)\)-partition-edge fault-tolerant paired 2-DPC \cite{2025Paired}. 
Recently, Lin et al. \cite{0Link} 
adopted the PEF model to analyze BCube's fault-tolerant Hamiltonian connectivity. They demonstrated that BCube remains Hamiltonian-connected when $e_i \leq n^i - 2$ for $1 \leq i \leq k$ and $4 \leq n \leq 9$, or $e_i \leq \lceil (n^i - 1)/2 \rceil (n - 1) - 1$ for $1 \leq i \leq k$ and $n \geq 10$, with the additional constraint $e_0 \leq n - 4$.
In summary, these studies indicate that the application of the PEF model substantially enhances the edge-fault-tolerance of DCNs.

Recent studies on network reliability have advanced beyond traditional edge connectivity by proposing various enhanced fault-tolerant metrics. Among these, matroidal connectivity and its conditional variant, grounded in matroid theory, have been introduced to evaluate the structural robustness of networks under failure conditions. Related research has covered typical network topologies, including star graphs\cite{ZHUANG2023114173}, regular networks\cite{zhang2024characterization}, Cayley graphs\cite{wang2025conditional}, and varietal hypercubes \cite{yang2025conditional}. On the other hand, edge-partition-based fault-tolerant analysis methods partition edges by dimension or structural units, offering new evaluation approaches for hypercubes, folded hypercubes, and balanced hypercubes in high-failure-rate scenarios \cite{chen2024novel,liu2024enhancing}. Investigations into alternating group graphs \cite{zhang2022high} and arrangement graphs \cite{li2025enabling} further reveal the inherent high fault tolerance of specific symmetric structures. In the context of data center networks, novel reliability indicators \cite{zhuang2024novel} and matroidal connectivity analyses for link/switch failures \cite{2025Link} provide crucial  support for network architecture design. Collectively, these studies establish a theoretical foundation for constructing a partitioned edge fault model in BCube networks and for analyzing disjoint path cover fault tolerance in this work.


Lin et al. \cite{0Link} raised some open questions, one of which is to consider the fault-tolerant $m$-DPC of BCube under the PEF model. In our article, we address this problem by establishing the edge-fault-tolerance of paired 2-DPC in BCube under the PEF model.
Our contributions are as follows:

1. Under the PEF model, we prove that BCube always contains a paired 2-DPC connecting each two distinct source-destination pairs, provided the number of faulty edges does not exceed the following bounds:
\begin{itemize}
    \item $\sum_{i=1}^k(\max\{0, n^i-5\})+n-5$ for $n=7$
    \item $\sum_{i=1}^k(\max\{0, n^i-5\})+n-4$ for $n\in \{4,5,6,8,9\}$
    \item $\sum_{i=1}^k(\left\lceil \frac{n^i - 1}{2} \right\rceil \cdot(n-2) - \frac{3}{2}n)+n-4$ for $n\geq10$.
\end{itemize}

2. To the best of our knowledge, our result is the first exponential bound for edge-fault-tolerance of 2-DPC in BCube.

This article is organized as follows. Section 2 introduces the basic symbols and definitions. Section 3 proves the existence of 2-DPC in BCube under the PEF model conditions. 
Finally, Section 4 summarizes the entire paper.

\section{Preliminaries} 
\label{preliminary}

\subsection{Terminology and Notation} 

\begin{table}
    \centering
    \caption{Notations}
    \scriptsize  
    \begin{tabular}{cl} 
        \toprule
        Symbol & Meaning \\
        \midrule
        $\langle n \rangle$ & An integer set $\{0,1,2,\ldots,n\}$  \\
        $G=(V,E)$ & A graph with node set $V$ and edge set $E$ \\
        $V(G)$ & A node set of graph $G$ \\
        $E(G)$ & An edge set of graph $G$ \\
        $G-F$ & \makecell[l]{A graph with $V(G-F)=V(G)$ and \\$E(G-F)=E(G)\setminus F$} \\ 
        $BC_{n,k}$ & The logic graph of BCube$(n,k)$ \\
        $BC[m]$ & A subgraph of $BC_{n,k}$ isomorphic to $BC_{n,k-1}$ \\
        $BC[\Omega]$ & \makecell[l]{The subgraph induced by the vertices of \\$BC[m]$ for all $m\in\Omega$} \\
        $F_i$ & The set of faulty edges in $i$-th dimension ($F_i \subseteq E_i$) \\
        $r_i$ & The $i$-th smallest number in $\{|F_0|,...,|F_k|\}$ \\
         $E(l_1,l_2)$ & The set of edges between  \( BC[l_1] \) and \( BC[l_2] \)  \\
        \bottomrule
    \end{tabular}
    \label{tab:notations}
\end{table}

Some basic notations used in this paper are listed in Table \ref{tab:notations}. Next we present the formal definition of 2-DPC.

\begin{definition} (\cite{2024Paired}): Let \( S = \{s_1, s_2\} \) and \( T = \{t_1, t_2\} \) be the sets of two fault-free source and destination nodes respectively, with \( S \cap T = \emptyset \). A paired 2-disjoint path cover (paired 2-DPC) is a set of two disjoint fault-free paths that connect each pair of \( s_i \) and \( t_i \) for \( i \in \{1, 2\} \), and cover all the fault-free nodes in \( G \).
\end{definition}

 A graph \( G \) is called \emph{paired 2-disjoint path coverable}, if for each source set \( S=\{s_1,s_2\} \) and destination set \( T=\{t_1,t_2\} \) with $S \cap T = \emptyset$, there exists a paired 2-DPC in \( G \). A graph \( G \) is called \emph{t-edge fault-tolerant} paired 2-disjoint path coverable, if for every fault edge set \( F \) with $|F|\leq t$, the graph \( G - F \) is still paired 2-disjoint path coverable. Let $P := s_1u_0u_1...u_mt_1$ is a path from $s_1$ to $t_1$ passing through node $u_1...u_m$.

\subsection{BCube}

BCube is a server-centered data center network. We view switches transparent, because they only serve as crossbars for transferring packets between servers. The logic graph of BCube(n,k) is obtained by taking servers as nodes and joining two serves by an edge if they are connected to the same switch. The resulting logical graph of BCube(n,k) is denoted as $BC_{n,k}$. The definition of \( BC_{n,k} \) is as follows:

\begin{definition} (\cite{2023Fault}): For \( n \geq 2 \) and \( k \geq 0 \), \( BC_{n,k} \) is a simple undirected graph, where the set of nodes is \(  \{a_k a_{k-1} ... a_0 \colon\, a_i \in \left<n-1\right>, i \in \left<k\right>\} \), and two nodes $u=a_k a_{k-1} ... a_0$ and $v=b_k b_{k-1} ... b_0$ are connected by an edge if there exists exactly one integer $i$ such that $a_i \neq b_i$ and $a_j=b_j$ for all $j\in{\langle k \rangle\setminus i}$.
\end{definition}

According to the above definition, it can be concluded that the number of nodes is \( |V(BC_{n,k})| = n^{k+1} \), and the number of edges is \( |E(BC_{n,k})| = \frac{n^{k+1}  (n-1)  (k+1)}{2} \). 
The set of edges in the \( i \)-th dimension is defined as \( E_i(BC_{n,k}) = \{(a_k a_{k-1} ... a_i ... a_0, a_k a_{k-1} ... b_i ... a_0)\in E(BC_{n,k})\colon\,a_i\neq b_i \} \). We abbreviate $E_i(BC_{n,k})$ as $E_i$ if the network $BC_{n,k}$ is clear in the context. Let $F$ be a fault edge set and \( F_i = F \cap E_i \) be the set of faulty edges in the \( i \)-th dimension.
Set $\{r_0,r_1,...r_{k}\}=\{|F_i|\colon\,i \in \left<k\right>\}$ such that $r_{k} \geq r_{k-1} \geq\cdot \geq r_0$. 
The fault edge set $F$ is called a \emph{partitioned edge fault set} (PEF set) if and only if $r_j \leq f(j) < |E_i|=\frac{n^{k+1}\cdot(n-1)}{2}$ for each $j$, where $f(j)$ is a function of $j$ or a fixed value\cite{2023An}.

Without loss of generality, assume that $k$ is the dimension with the most faulty edges. Wer partition all the nodes of $BC_{n,k}$ into $k+1$ subsets according to the $k$-th dimension. For $m\in \langle n-1\rangle$, let $BC_{n,k}[m]$ be the subgraph induced by the nodes $\{m a_{k-1} ... a_0|a_i \in \left<n-1\right>$ for $i \in \langle k-1 \rangle\}$. The partitioned subgraph \( BC_{n,k}[m] \) is isomorphic to \( BC_{n,k-1} \). In the subsequent discussion, we abbreviate \( BC_{n,k}[m] \) as $BC[m]$ when the specific scale of the BCube is not emphasized. For the node $u=a_ka_{k-1}...a_0$, let $l_u=a_k$. Thus $u\in{V(BC[l_u])}$. For a node \( u\in V(BC[m]) \), the unique neighbor node of \( u \) in the subgraph \( BC[m'] \) is denoted by \( n^{m'}(u) \). 
Let \( E(l_1, l_2) \) represent the set of edges between the subgraphs \( BC[l_1] \) and \( BC[l_2] \). Clearly, \( |E(l_1, l_2)| = n^k \). 
For a subset $\Omega \subseteq \left< n - 1 \right>$, denote $BC[\Omega]$ the subgraph induced by the nodes \( \{i a_{k-1} ... a_{j} ... a_0 | i \in \Omega \text{ and } a_j \in \left<n-1\right>\) for \(j \in \left<k-1\right>\}\). For source set $\{s_1,s_2\}$ and destination set $\{t_1,t_2\}$, let $\Omega_1 = \left<n-1\right>\setminus{\{l_{s_1}\}},\Omega_2 = \left<n-1\right>\setminus{\{l_{t_1}\}},\Omega_3= \left<n-1\right>\setminus{\{l_{s_1},l_{s_2}\}}$, and $\Omega_4 = \left<n-1\right>\setminus{\{l_{s_1},l_{t_1}\}}$.

Figure \ref{fig:BCube3,1} shows the partition of $BC_{3,1}$ along the 1-st dimension, where $BC[0]$ is the subgraph induced by $\{00,01,02\}$, $BC[1]$ is the subgraph induced by $\{10,11,12\}$, and $BC[2]$ is the subgraph induced by $\{20,21,22\}$. Each $BC[i]$ is isomorphic to the complete graph $K_3$. For each node in $BC[i]$, it is adjacent to exactly one node in $BC[j]$ for $j\neq i$. For example, the node $00$ in $BC[0]$ is adjacent to the node $10$ in $BC[1]$ and node $20$ in $BC[2]$. $E_0$ is the set of edges inside these subgraphs, and $E_1$ is the set of edges between these subgraphs.


\begin{figure}[H]
    \centering
    \includegraphics[width=3in]{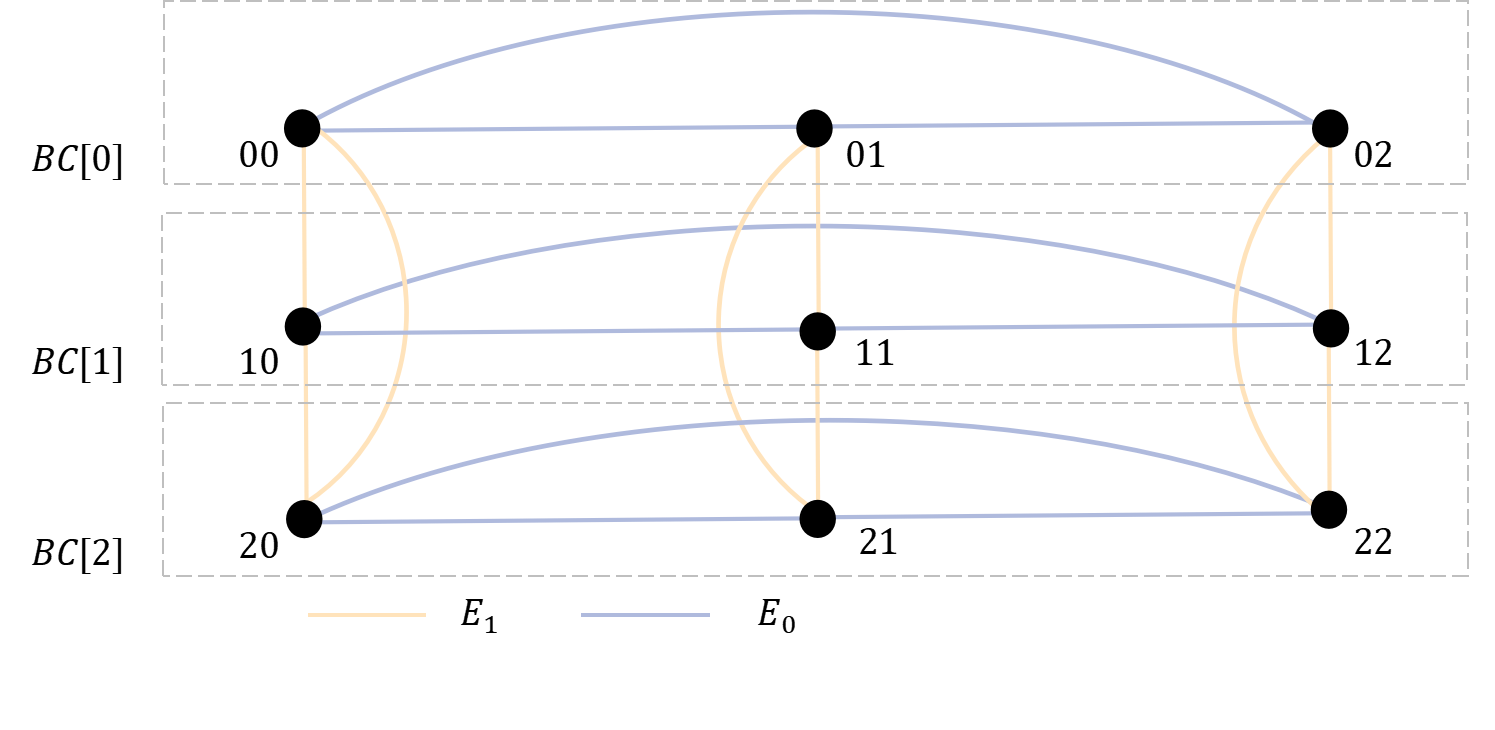}
    \caption{Illustration of $BC_{3,1}$}
\label{fig:BCube3,1}
    \end{figure} 

\section{Paired 2-disjoint path covers in BCube}

\begin{lemma} (\cite{Hsu2004Fault})\label{Kn-H}
Let $K_n$ be the complete graph with n nodes. For $n \geq 4$, $K_n$ is $(n - 4)$-edge fault-tolerant Hamiltonian-connected.
\end{lemma}

Next, we present a result on fault-tolerant 2-DPC in complete graphs. It can be derived through careful analysis.

\begin{claim}
\label{Kn-(n-4)}
For $n\geq4$ and $n\neq7$, if $|F|\leq n-4$, then $K_n - F$ is paired 2-disjoint path coverable; for $n=7$, if $|F|\leq n-5$, then $K_7 - F$ is paired 2-disjoint path coverable.
\end{claim}

Note that $BC_{3,k}$ is isomorphic to $3$-ary $(k + 1)$-cube. The following lemma is a partial result obtained from \cite[Theorem 2]{2024Paired}.
\begin{lemma}(see \cite{2024Paired})  Let $F$ be any faulty
edge set of $BC_{3,k}$ for
 $k \geq 1$. If the edges in \( F \) satisfy the following conditions:

(1) \( f(0) = 0,f(1)=1 \);

(2) \( f(i)=3^i - 2 \; \text{ for each } i \in \{2,3,...,k\}\);

(3) $ |F|\leq\frac{3^{k+1}-3}{2}-2k $,

then $BC_{3,k}-F$ is paired 2-disjoint path coverable. 
\end{lemma}

Next we give the definition of partitioned edge fault (PEF) model used in our paper:
\begin{definition}\label{PEF}
 For \( n \geq 4, k \geq 0 \), a fault edge set $F$ of \( BC_{n,k} \) is called an f-PEF if the edges in \( F \) satisfy the following conditions:

(1) $ f(0) \leq \left\{
\begin{aligned}
& n-5 & \text{ if } n = 7\\
& n-4  & \text{ otherwise }
\end{aligned}
\right.
$

(2) For \( 1 \leq i \leq k \):
$$ f(i) \leq \left\{
\begin{aligned}
& \max\{0, n^i-5\} & \text{ if } 4 \leq n \leq 9\\
& \left\lceil \frac{n^i - 1}{2} \right\rceil \cdot(n-2) - \frac{3n}{2}  & \text{ if } n \geq 10.
\end{aligned}
\right.
$$
\end{definition}

\begin{lemma}\label{disjoint}
For \( n \geq 4, k \geq 1 \), let the fault edge set \( F \) be an $f$-PEF in \( BC_{n,k} \). Let $P_1$ and $P_2$ be paired 2-DPC in $BC[m]$ for some integer $m$, and $\Omega \subseteq \left< n - 1 \right>$ with $|\Omega| \geq n-2$ and $m\notin \Omega$. Then \( E(P_1) \cup E(P_2) \) contains at least one edge \( (x, x^*) \) such that there exist distinct \( l_1, l_2 \in \Omega\) satisfying 
\( (x, n^{l_1}(x)) \notin F \) and \( (x^*, n^{l_2}(x^*)) \notin F \).
\end{lemma}
\begin{proof}
In \( BC[m] \), \( P_1 \) and \( P_2 \) together contain \( n^k - 2 \) edges. Therefore, there are \( \left\lceil \frac{n^k - 2}{2} \right\rceil \) pairs of disjoint edges. The proof is as follows:

(1) For \( 4 \leq n \leq 9 \):

Since \( |F_k| \leq n^k - 5 \), we have \( \frac{n^k - 2}{2} \cdot |\Omega| - F_k \geq \frac{n^k - 2}{2} (n - 2) - (n^k - 5)= (n^k - 2) \left( \frac{1}{2}(n - 2) - 1 \right) + 3= (n^k - 2) \left( \frac{n}{2} - 2 \right) + 3 \). The function $(n^k - 2) \left( \frac{n}{2} - 2 \right) + 3$ is increasing with $n$ and $k$. For \( n = 4 \), \( k = 1 \), we have \(
(4 - 2) \cdot \left( \frac{4}{2} - 2 \right) + 3 = 3 > 1\)

(2) For \( n \geq 10 \):

Since \( |F_k| \leq \left\lceil \frac{n^k - 1}{2} \right\rceil (n - 2) - \frac{3n}{2}\), we have\( \frac{n^k - 2}{2} \cdot |\Omega| - F_k \geq \frac{n^k - 2}{2} (n - 2) - \frac{n^k}{2} (n - 2) + \frac{3n}{2} = \frac{n^k}{2} (n - 2) - (n - 2) - \frac{n^k}{2} (n - 2) +\frac{3n}{2} = 2 + \frac{n}{2}> 1\)
 
Therefore, at least one pair of edges \( (x, x^*) \) exists such that \( (x, n^{l_1}(x)) \) and \( (x^*, n^{l_2}(x^*)) \notin F \), where $l_1, l_2\in \langle n\rangle\setminus m$ are distinct.
\end{proof}

\begin{lemma}(\cite{0Link})\label{hcbc}
 For \( n \geq 4 \), \( k \geq 0 \), let \( F \) be the set of faulty edges in \( BC_{n,k} \), and suppose the edges in \( F \) satisfy the following conditions:
 
(1) \( f(0) = n-4 \)

(2) For \( 1 \leq i \leq k \):
$$ f(i) \leq \left\{
\begin{aligned}
&n^i - 2 & \text{ if } 4 \leq n \leq 9\\
& \left\lceil \frac{n^i - 1}{2} \right\rceil \cdot(n-1) - 1  & \text{ if } n \geq 10.
\end{aligned}
\right.
$$
In this case, \( BC_{n,k} - F \) is Hamiltonian-connected.
\end{lemma}

Since the number of faulty edges permitted in Lemma \ref{hcbc} exceeds the maximum number of faulty edges that can be accommodated in $f$-PEF, the following conclusion can be obtained:

\begin{corollary}\label{H-path}
For \( n \geq 4 \), \( k \geq 0 \), suppose the fault edge set \( F \) in \( BC_{n,k} \) is an f-PEF. Then \( BC_{n,k} - F \) is Hamiltonian-connected.
\end{corollary}

Now we define the \emph{contracted graph} $T(\Omega, F)$ of $BC_{n,k}$ for the fault edge set $F$ and $\Omega \subseteq \left< n - 1 \right>$ with $|\Omega| \geq 2$ as follows.  $A$  node $c_i$ in $T(\Omega, F)$ corresponds to the subgraph $BC[i]$, and $c_i$ is adjacent to $c_j$ if and only if their corresponding subgraphs $BC[i]$ and $BC[j]$ are connected by at least three fault-free edges under $F$.

\begin{lemma}(\cite{0Link})\label{sub-connected-l}
Let $\Omega \subseteq \left< n - 1 \right>$ with $|\Omega| \geq 2$, and $F$ be the fault edge set of $BC_{n,k}$ such that $BC[\ell] - F$ is Hamiltonian-connected
 for each $\ell \in \Omega$. Let $s \in V(BC[i])$ and $t \in V(BC[j])$, where $i, j \in \Omega$ are distinct. If there exists a Hamiltonian path from $c_i$ to $c_j$ in $T(\Omega,F)$, then so is the existence of an ($s$,$t$)-Hamiltonian
 path in $BC[\Omega] - F$.
\end{lemma}

\begin{corollary}\label{sub-connected}
 Let $\Omega \subseteq \left< n - 1 \right>$ with $|\Omega| \geq n-2$.
 Assume that $F$ is the fault edge set of $BC_{n,k}$ for $n \geq 4$ and $k\geq 1$. If $F$ is an f-PEF, then for any two distinct integers $i, j  \in \Omega$, every two nodes $s \in V(BC[i])$ and $t \in V(BC[j])$, there is an ($s$,$t$)-Hamiltonian
 path in $BC[\Omega] - F$.
\end{corollary}
\begin{proof}
In $BC_{n,k}$, the subgraphs are partitioned accor    ding to the dimension with the highest number of faulty edges. Between any two subgraphs, the number of edges is $n^k$. 
    
When $4 \leq n \leq 9$, $f(k)=\max\{n^k - 5,0\}$, where $f(k)$ denotes the total number of faulty edges across $BC[m],m\in \left< n - 1 \right>$. In the fault-free scenario, the number of edges between subgraphs is $n^k$. There are at least four fault-free edges between any two subgraphs, and thus $T(\Omega,F)$ is a complete graph. Therefore, there exists a Hamiltonian path from $i$ to $j$ in $T(\Omega,F)$.

When $n \geq 10$, $f(k)=\left\lceil \frac{n^k - 1}{2} \right\rceil (n-2) - \frac{3n}{2}$. By the Lemma \ref{Kn-H}, $T(\Omega, F)$ is $(|\Omega| - 4)$-Hamiltonian-connected. If there exist exactly \( |\Omega| - 3 \) paired of subgraphs such that the number of faulty edges between any two subgraphs is $n^k-2$, then \( T(\Omega,F) \) will contain \( |\Omega| - 3 \) faulty edges, in which case \( T(\Omega,F) \) is not Hamiltonian-connected.  
$(|\Omega|-3)\cdot(n^k-2)-\lceil\frac{n^k-1}{2}\rceil\cdot(n-2)+\frac{3n}{2} \geq (n-2-3)\cdot(n^k-2)-\frac{n^k}{2}\cdot(n-2)+\frac{3n}{2} = (\frac{n-8}{2})\cdot n^k-\frac{n}{2}+10$. When $k \geq 1$ and $n \geq 10$, $(\frac{n-8}{2})\cdot n^k-\frac{n}{2}+10$ is increasing. When $k =1,n=10$, $(\frac{n-8}{2})\cdot n^k-\frac{n}{2}+10 = 15> 1$. Then, $(|\Omega|-3)\cdot(n^k-2)-\lceil\frac{n^k-1}{2}\rceil\cdot(n-2)+\frac{3n}{2} > 0$. Thus, $T(\Omega,F)$ remains Hamiltonian-connected. Thus, there must exist a Hamiltonian path from $c_i$ to $c_j$. 

According to the Corollary \ref{H-path} and Lemma \ref{sub-connected-l}, there exists an ($s$,$t$)-Hamiltonian path in $BC[\Omega] - F$.
\end{proof}


Next we give our main result about fault-tolerance for the paired 2-disjoint path covers of
BCube under the PEF model:

\begin{theorem}\label{main}
For $n\geq 4$ and $k\geq0$, if the fault edge set $F$ in \( BC_{n,k} \) satisfies the f-PEF condition, then \( BC_{n,k} - F \) is paired 2-disjoint path coverable.
\end{theorem}
\begin{proof}
We prove this theorem by induction on $k$. Since $BC_{n,0}$ is isomorphic to $K_n$, according to Claim \ref{Kn-(n-4)}, $BC_{n,0}$ is $(n-4)$-edge fault-tolerant paired 2-disjoint path coverable for $n\geq4$ and $n\neq7$. Moreover, $BC_{7,0}$ is $(n-5)$-edge fault-tolerant paired 2-disjoint path coverable. Thus Theorem~\ref{main} holds for $k = 0$.

For \( n = 4 \) and \( k =1 \), we conduct code-based verification using an exhauive brute-force method.\footnote{https://github.com/oliviacqq/bcube-2dpc} 
For the other cases where \( n = 4  ,  k \geq 2 \) or \( n \geq 5  ,  k \geq 1 \), the proof proceeds as follows.

First we show a claim which is useful for the induction.
\begin{claim}
    If the fault edge set $F$ in $BC_{n,k}$ is an $f$-PEF, then the fault edge set $F'$ inside $BC[m]$ ($BC[m]$ is isomorphic to $BC_{n,k-1}$) is also an $f$-PEF with respect to $n$ and $k-1$.
\end{claim}
\textit{Proof of Claim 2:}
  Let $F_i$ be the set of faulty edges in the $i$-th dimension of $BC_{n, k}$, and $F'_i$ be the set of faulty edges in the $i$-th dimension inside the subgraph $BC[m]$ for some $m\in \left< n - 1 \right>$.  
Assume that $|F_i|$ is the $a$-th smallest among $\{|F_0|,...,|F_{k-1}|\}$, and $|F_i'|$ is the $b$-th smallest among $\{|F_0'|,...,|F_{k-1}'|\}$. In order to prove $F'$ is an $f$-PEF with respect to $n$ and $k-1$, by definition of PEF set, we only need to show $|F'_i|\leq f(b)$.

If $b \geq a$, then $f(b) \geq f(a)$. Therefore, we have $|F_i'| \leq |F_i| \leq f(a) \leq f(b)$. 
Otherwise, there must exist an $F_l$ such that $|F_l|$ is the $p$-th smallest among $\{|F_0|,...,|F_{k-1}|\}$ for some $p \in [0,b]$, and $|F_l'|$ is the $q$-th smallest among $\{|F_0'|,...,|F_{k-1}'|\}$ for some $q \in [b+1,k-1]$. 
Thus, we have  $|F_i'| \leq |F_l'| \leq |F_l| \leq f(p) \leq f(b)$. It completes the proof of Claim 2. $\hfill\qed$

Assume that Theorem~\ref{main} holds for all $BC_{n,k'}$ where $k' < k$. We now prove the statement for $BC_{n,k}$. For each source set \( \{s_1, s_2\} \) and destination set \( \{t_1, t_2\} \), where these four nodes are all different, we want to find two disjoint fault-free paths that connect each pair of source $s_i$ and destination $t_i$ and cover all the nodes in $BC_{n,k}$. To prove this, we divide into the following cases:

\textbf{Case 1}: \( s_1, s_2, t_1, t_2 \) belong to the same subgraph \( BC[l_{s_1}] \).

By inductive hypothesis, we can construct a paired 2-DPC \( P_1 \) from $s_1$ to $t_1$ and \( P_2 \) from $s_2$ to $t_2$ in \( BC[l_{s_1}] - F\). According to Lemma \ref{disjoint}, in \( BC[l_{s_1}] - F\), there exists an edge \( (x, x^*) \in P_1\cup P_2\) such that \( \{ (x,n^{l_1}(x)),  (x^*, n^{l_2}(x^*))\} \notin F\) and $l_1\neq l_2$. By Corollary \ref{sub-connected}, in $BC[\Omega_1] - F$, where $\Omega_1=\langle n-1 \rangle \setminus \{l_{s_1}\}$, there exists a Hamiltonian path \( H_1 \) from $n^{l_1}(x)$ to $n^{l_2}(x^*)$. Therefore, we can construct a paired 2-DPC in \( BC_{n,k} -F\) formed by the paths \( P_1 \) and \( P_2 \cup H_1 \cup \{(x, n^{l_1}(x)), (x^*, n^{l_2}(x^*))\} \setminus \{(x, x^*)\} \), as shown in Figure \ref{fig:Case1 and Case2}(a).

\begin{figure}[H]
    \centering
    \includegraphics[width=0.55\linewidth]{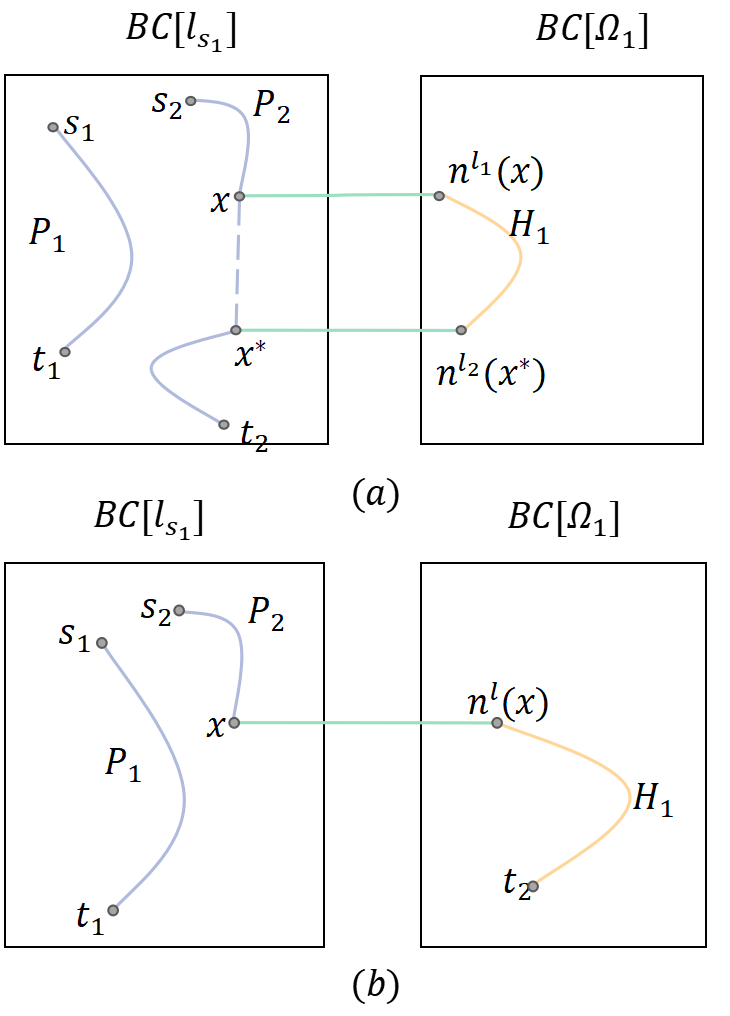} 
    \caption{Case 1 and Case 2 of Theorem \ref{main} ($\Omega_1 = \left<n-1\right>\setminus{l_{s_1}}$) }
\label{fig:Case1 and Case2}
    \end{figure} 

\textbf{Case 2}: Exactly three nodes from \( s_1, s_2, t_1, t_2 \) belong to the same subgraph. Without loss of generality, we assume that \( s_1, s_2, t_1 \) belong to the same subgraph \( BC[l_{s_1}] \).

Since \( BC[l_{s_1}] \) is a graph with \( n^k \) nodes, $BC[l_{s_1}] - \{s_1,s_2,t_1\}$ contains \( n^k - 3 \) nodes. These \( n^k - 3 \) nodes can establish connections to \( n - 2 \) distinct subgraphs except \( BC[l_{s_1}] \) and \( BC[l_{t_2}] \), yielding \( (n^k - 3)(n - 2) \) edges. Note that $F_k$ is the number of faulty edges across subgraphs $BC[m],m\in{\left<n-1\right>}$. If \((n^k - 3)(n - 2) - |F_k| > 0\), then there must exist a node \(  x \in V(BC[l_{s_1}])\setminus\{s_1, s_2, t_1\}\) such that \((x,n^{l}(x))\notin F\) for some \(l \in \left<n-1\right> \setminus\{l_{s_1}, l_{t_2}\} \).  


• For $n=4,k\geq2$, $(n^k - 3)(n - 2) - |F_k|=(4^k-3)*(4-2)-(4^k-5)=4^k-1\geq1$.

• For \( 5 \leq n \leq 9 \), \((n^k - 3)(n - 2) - |F_k| \geq
(n^k - 3)(n - 2) - (n^k - 5) = (n^k - 3)(n-3)+2 \geq 1\). 

• For \( n \geq 10 \), \((n^k - 3)(n - 2) - |F_k| \geq (n^k - 3)(n - 2) - \frac{n^k}{2}  (n - 2) + \frac{3n}{2} = n^k\cdot\frac{n-2}{2}-\frac{3n}{2}+6\). When $n \geq 10$, the function $g(n,k) = n^k\cdot\frac{n-2}{2}-\frac{3n}{2}+6$ is increasing. For \( n = 10 \) and \( k = 1 \), we get \(n^k\cdot\frac{n-2}{2}-\frac{3n}{2}+6 = 31 > 1\).

Therefore, we can find \( x \in V(BC[l_{s_1}])\setminus\{s_1, s_2, t_1\}\) such that \((x,n^{l}(x))\notin F\) for some \(l \in \left<n-1\right> \setminus\{l_{s_1}, l_{t_2}\} \). In \( BC[l_{s_1}] -F\), by inductive hypothesis, we can construct a paired 2-DPC \( P_1 \) from $s_1$ to $t_1$ and \( P_2 \) from $s_2$ to $x$. According to Corollary \ref{sub-connected}, there exists a Hamiltonian path \( H_1 \) connecting \( n^{l}(x) \) and \( t_2 \) in $BC[\Omega_1]-F$. Thus, we can construct a paired 2-DPC in \( BC_{n,k} -F\), formed by  \( P_1 \) and \( P_2 \cup H_1 \cup \{(x, n^{l}(x))\} \), as shown in Figure \ref{fig:Case1 and Case2}(b).

\textbf{Case 3}: Two nodes from \( s_1, s_2, t_1, t_2 \) belong to the same subgraph.

Case 3.1: \( \{s_1, s_2\}\) or \( \{t_1,t_2\} \) belong to the same graph. By symmetry, it suffices to prove the case where both $s_1$ and $s_2$ belong to the subgraph $BC[l_{s_1}]$.

\begin{figure}[H]
    \centering
    \includegraphics[width=0.65\linewidth]{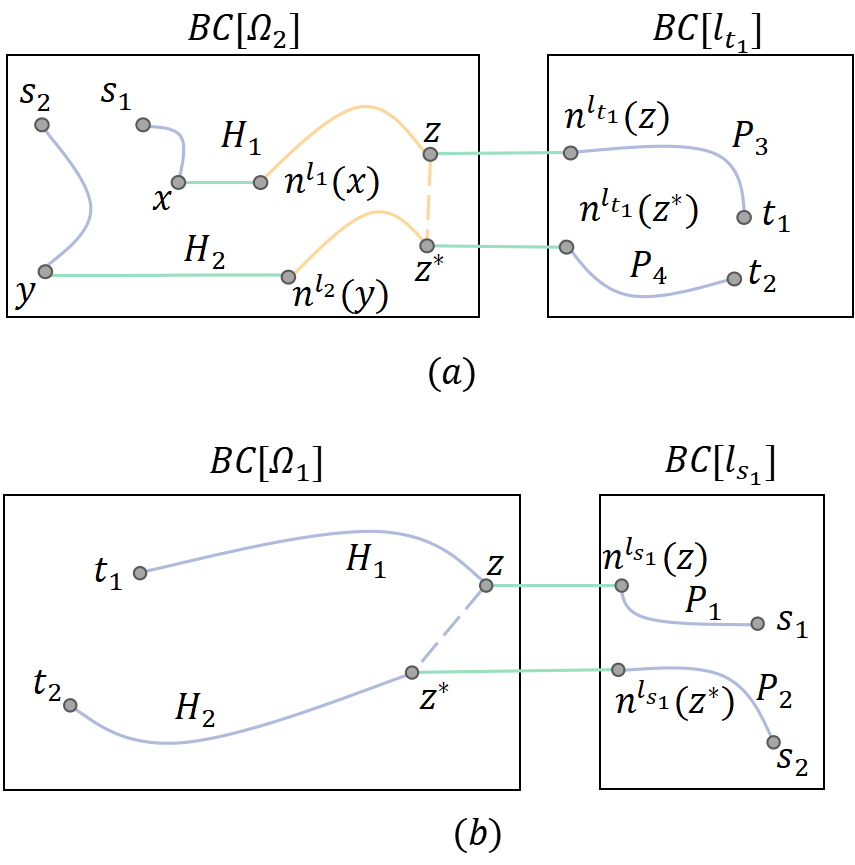}
    \caption{Case 3.1 of Theorem \ref{main} ($\Omega_1 = \left<n-1\right>\setminus{\{l_{s_1}\}},\Omega_2 = \left<n-1\right>\setminus{\{l_{t_1}\}}$)}
    \label{fig:Case3.1}
\end{figure} 
(1) $t_1$ and $t_2$ also belong to the same subgraph $BC[l_{t_1}]$ (i.e. $l_{t_1} = l_{t_2}$).

In \( BC[l_{s_1}] \), apart from $s_1$ and $s_2$, there are $n^k-2$ nodes that connect to $n-2$ subgraphs excluding both $BC[l_{s_1}]$ and $BC[l_{t_1}]$.

• For \( n = 4,k \geq 2 \), $(n^k-2)(n-2)-|F_k|\geq(n^k-2)(n-2)-(n^k-5) = (n^k-2)\cdot(n-3)+3 \geq 17 $.

• For \( 5 \leq n \leq 9 \), $(n^k-2)(n-2)-|F_k|\geq(n^k-2)(n-2)-(n^k-5) = (n^k-2)\cdot(n-3)+3 \geq 9 $.

• For \( n \geq 10 \), $(n^k-2)(n-2)-|F_k|\geq(n^k-2)(n-2)-\frac{n^k}{2}(n-2)+\frac{3n}{2} = n^k(\frac{n}{2}-1)-\frac{n}{2}+4 \geq 39 $. 

Thus, in \( BC[l_{s_1}] \), there exists $x\in V(BC[l_{s_1}])\setminus\{s_1,s_2\}$ such that $(x,n^{l_1}(x)) \notin{F}$ for some $l_1\in{\left<n-1\right>}\setminus\{l_{s_1},l_{t_1}\}$. 

Similarly, we compute 

• For \( n = 4,k \geq 2 \), $(n^k-3)(n-3)-|F_k|\geq(n^k-3)(n-3)-(n^k-5) = (n^k-3)\cdot(n-4)+2 \geq 2 $.

• For \( 5 \leq n \leq 9 \), $(n^k-3)(n-3)-|F_k|\geq(n^k-3)(n-3)-(n^k-5) = (n^k-3)\cdot(n-4)+2 \geq 4 $.

• For \( n \geq 10 \), $(n^k-3)(n-3)-|F_k|\geq(n^k-3)(n-3)-\frac{n^k}{2}(n-2)+\frac{3n}{2} = n^k(\frac{n}{2}-2)-\frac{3n}{2}+9 \geq 24 $. 

Thus, there exists $y\in V(BC[l_{s_1}])\setminus\{s_1,s_2,x\}$ such that $(y,n^{l_2}(y)) \notin{F}  $ for some $l_2\in{\left<n-1\right>}\setminus\{l_{s_1},l_{t_1},l_1\}$.

By inductive hypothesis, we can construct a paired 2-DPC $P_1$ from $s_1$ to $x$ and $P_2$ from $s_2$ to $y$ in $BC[l_{s_1}]-F$. Recall that $\Omega_4 = \left<n-1\right>\setminus{\{l_{s_1},l_{t_1}\}}$ and $\Omega_2 = \left<n-1\right>\setminus{\{l_{t_1}\}}$. According to Corollary \ref{sub-connected}, in $BC[\Omega_4]-F$, there exists a Hamiltonian path $H$ from \( n^{l_1}(x) \) to \( n^{l_2}(y) \). In \( BC[\Omega_2] \), there exists a Hamiltonian path $H'=P_1\cup P_2\cup H \cup \{(x,n^{l_1}(x)),(y,n^{l_2}(y))\}$. In $H'$, $t_1$ is connected to the other $n-1$ subgraphs by $n-1$ edges, and similarly, $t_2$ is connected to the other $n-1$ subgraphs by $n-1$ edges.

• For \( n = 4,k \geq 2 \), $\lceil\frac{(n-1)n^k-1}{2}\rceil-2(n-1)-|F_k| \geq \frac{(n-1)n^k-1}{2}-2(n-1)-(n^k-5)=\frac{n-3}{2}n^k-2n+\frac{13}{2}\geq \frac{13}{2}$.

• For \( 5 \leq n \leq 9 \), $\lceil\frac{(n-1)n^k-1}{2}\rceil-2(n-1)-|F_k| \geq \frac{(n-1)n^k-1}{2}-2(n-1)-(n^k-5)=\frac{n-3}{2}n^k-2n+\frac{13}{2}\geq \frac{3}{2}$.

• For \( n \geq 10 \), $\lceil\frac{(n-1)n^k-1}{2}\rceil-2(n-1)-|F_k| \geq \frac{(n-1)n^k-1}{2}-2(n-1)-\frac{n^k}{2}(n-2)+\frac{3n}{2}=\frac{1}{2}n^k-\frac{n}{2}+\frac{3}{2}\geq10\cdot\frac{1}{2}-10\cdot\frac{1}{2}+\frac{3}{2}>1$.

Thus, there exists an edge \( (z, z^*)\in H' \) such that \( (z, n^{l_{t_1}}(z)) , (z^*, n^{l_{t_1}}(z^*)) \notin F\) and \(n^{l_{t_1}}(z),n^{l_{t_1}}(z^*) \notin\{t_1,t_2\} \).  By deleting the edge $(z, z^*)$ from $H'$, we partition it into two disjoint subpaths: $H_1$ from $s_1$ to $z$ and $H_2$ from $s_2$ to $z^*$. In \( BC[l_{t_1}]-F \), we can construct a paired 2-DPC $P_3$ from $n^{l_{t_1}}(z)$ to $t_1$ and $P_4$ from $n^{l_{t_1}}(z^*)$ to $t_2$. Therefore, in \( BC_{n,k}-F \), we can construct a paired 2-DPC formed by \( P_3 \cup H_1 \cup \{(z,n^{l_{t_1}}(z))\}\) and \(P_4\cup H_2\cup \{(z^*,n^{l_{t_1}}(z^*))\}  \), as shown in Figure \ref{fig:Case3.1}(a).

(2) $t_1$ and $t_2$ belong to different subgraphs (i.e. \( l_{t_1} \neq l_{t_2} \)).

According to Corollary \ref{sub-connected}, in $BC[\Omega_1]$, there exists a Hamiltonian path $H$ from \( t_1 \) to \( t_2 \). 

In \( H \), $s_1$ is connected to the other $n-1$ subgraphs by $n-1$ edges, and similarly, $s_2$ is connected to the other $n-1$ subgraphs by $n-1$ edges.

• For \( n = 4,k \geq 2 \), $\lceil\frac{(n-1)n^k-1}{2}\rceil-2(n-1)-|F_k| \geq \frac{(n-1)n^k-1}{2}-2(n-1)-(n^k-5)=\frac{n-3}{2}n^k-2n+\frac{13}{2}\geq \frac{13}{2}$.

• For \( 5 \leq n \leq 9 \), $\lceil\frac{(n-1)n^k-1}{2}\rceil-2(n-1)-|F_k| \geq \frac{(n-1)n^k-1}{2}-2(n-1)-(n^k-5)=\frac{n-3}{2}n^k-2n+\frac{13}{2}\geq \frac{3}{2}$.

• For \( n \geq 10 \), $\lceil\frac{(n-1)n^k-1}{2}\rceil-2(n-1)-|F_k| \geq \frac{(n-1)n^k-1}{2}-2(n-1)-\frac{n^k}{2}(n-2)+\frac{3n}{2}=\frac{1}{2}n^k-\frac{n}{2}+\frac{3}{2}\geq10\cdot\frac{1}{2}-10\cdot\frac{1}{2}+\frac{3}{2}>1$.

Thus, there exists \( (z, z^*)\in H \) such that \((z, n^{l_{s_1}}(z))(z^*, n^{l_{s_1}}(z^*)) \notin F \) and \( n^{l_{s_1}}(z),n^{l_{s_1}}(z^*) \notin\{s_1,s_2\} \).  By deleting the edge $(z, z^*)$ from $H$, we partition it into two disjoint subpaths: $H_1$ from $t_1$ to $z$ and $H_2$ from $t_2$ to $z^*$. By inductive hypothesis, in \( BC[l_{s_1}]-F \), we can construct a paired 2-DPC $P_1$ from $s_1$ to $n^{l_{s_1}}(z)$ and $P_2$ form $s_2$ to $n^{l_{s_1}}(z^*)$. Therefore, in \( BC_{n,k}-F \), we can construct a paired 2-DPC formed by \( P_1 \cup H_1\cup \{(z,n^{l_{s_1}}(z))\}\) and \(P_2 \cup H_2 \cup \{(z^*,n^{l_{s_1}}(z^*))\} \), as shown in the Figure \ref{fig:Case3.1}(b).

Case 3.2: \( s_i, t_j \) belong to the same subgraph.

(1) If \( i \neq j \), assume without loss of generality that  \( s_1, t_2 \) belong to the same subgraph. By the symmetry of $s_2$ and $t_2$, it can be reduced to the Case 3.1, where \( s_1, s_2 \) belong to the same subgraph.


(2) If \( i = j \), assume without loss of generality that \( s_1, t_1 \) belong to the same subgraph \( BC[l_{s_1}] \). According to Corollary \ref{H-path}, in $BC[l_{s_1}]-F$, there exists a Hamiltonian path $H_1$ from \( s_1 \) to \( t_1 \). If $l_{s_2} \neq l_{t_2}$, according to Corollary \ref{sub-connected}, in $BC[\Omega_1]-F$, there exists a Hamiltonian path $H_2$ from \( s_2 \) to \( t_2 \). Therefore, $H_1$ and $H_2$ are a paired 2-DPC in \( BC_{n,k} - F\), as shown in the Figure \ref{fig:Case3.2}(a).

If $l_{s_2} = l_{t_2}$, according to Corollary \ref{H-path}, in $BC[l_{s_2}]-F$, there exists a Hamiltonian path $H_2$ from \( s_2 \) to \( t_2 \).

• For $n\geq4,k\geq2$ or $n\geq5,k\geq1$, we have $\lceil\frac{n^k-1}{2}\rceil(n-2)-|F_k| \geq \frac{n^k-1}{2}(n-2)-n^k+5=n^k\cdot\frac{n-4}{2}-\frac{n}{2}+6>=4$.

• For $n\geq10$, we have $\lceil\frac{n^k-1}{2}\rceil(n-2)-|F_k| \geq \frac{n^k-1}{2}(n-2)-\frac{n^k}{2}(n-2)+\frac{3n}{2}+1=n+2>=12$.

Thus, there exists \( (z, z^*)\in H_2 \) such that \( (z, n^{l_1}(z)) , (z^*, n^{l_2}(z^*)) \notin F\) for some distinct \(l_1,l_2 \in \left<n-1\right>\setminus\{l_{s_1},l_{s_2}\}\). According to Corollary \ref{sub-connected}, in $BC[\Omega_3]$, there exists a Hamiltonian path $H_3$ from \( n^{l_1}(z) \) to \( n^{l_2}(z^*) \) where $\Omega_3= \left<n-1\right>\setminus{\{l_{s_1},l_{s_2}\}}$. Therefore, we can construct a paired 2-DPC in \( BC_{n,k}-F \), formed by \( H_1 \) and \( H_2\cup H_3\cup \{(z, n^{l_1}(z)),(z^*, n^{l_2}(z^*))\} \setminus \{(z,z^*)\}\), as shown in the Figure \ref{fig:Case3.2}(b).
\begin{figure}[H]
    \centering
    \includegraphics[width=0.3\textwidth]{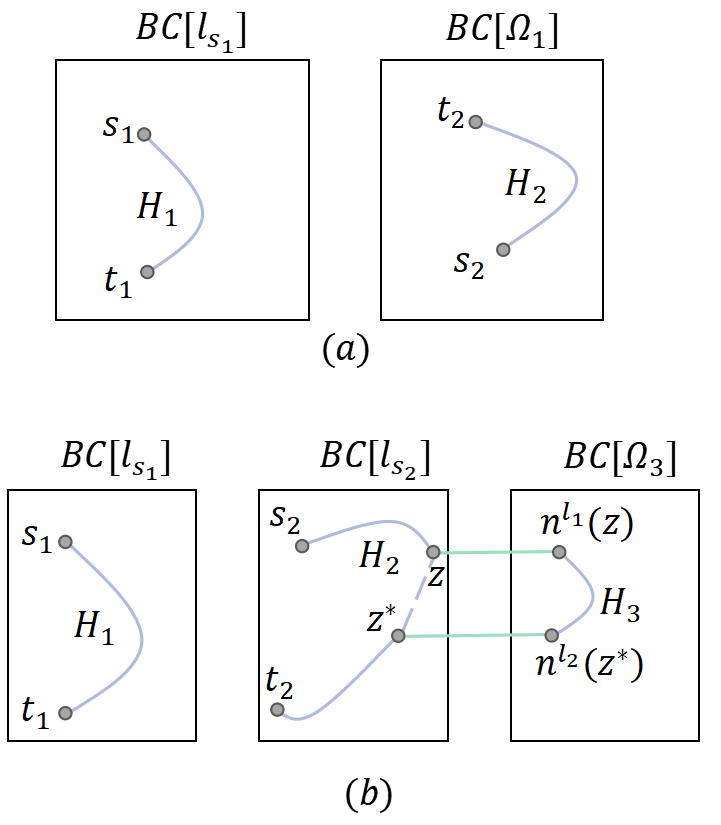}  
    \caption{
    Case 3 of Theorem \ref{main} with \( l_{t_1} \neq l_{t_2} \) and $n \geq 10$ ($\Omega_4 = \left<n-1\right>\setminus{\{l_{s_1},l_{t_1}\}},\Omega_2 = \left<n-1\right>\setminus{\{l_{t_1}\}}$)}
    \label{fig:Case3.2}
\end{figure}

\textbf{Case 4}: \( s_1, s_2, t_1, t_2 \) belong to four distinct subgraphs.

\begin{figure}[H]
    \centering
    \includegraphics[width=3in]{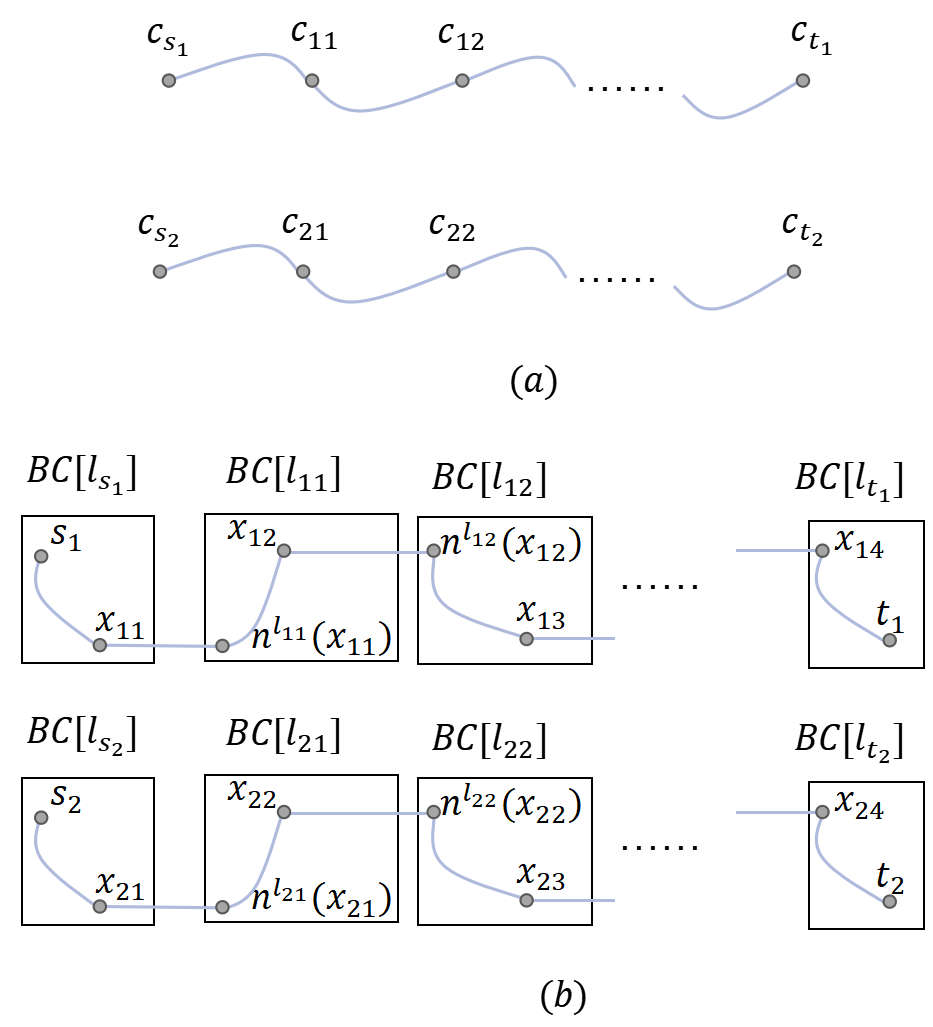}
    \caption{\( s_1, s_2, t_1, t_2 \) belong to different subgraphs}
    \label{fig:Case4}
\end{figure} 

We define the contracted graph $C$ of $BC_{n,k}$.  $A$  node $c_i$ in $C$ corresponds to the subgraph $BC[i]$, and $c_i$ and $c_j$ are adjacent if and only if their corresponding subgraphs are connected by at least two fault-free edges in $F$. Let $c_{s_1}, c_{s_2}, c_{t_1}$ and $c_{t_2} $ be the nodes in $C$ corresponding to the subgraphs $BC[l_{s_1}],BC[l_{s_2}],BC[l_{t_1}],BC[l_{t_2}]$ respectively, as shown in the Figure \ref{fig:Case4}(a).

When $n\geq4,k\geq2$ or $n\geq5,k\geq1$, $|F_k|\leq n^k - 5$, where $|F_k|$ denotes the total number of faulty edges across subgraphs. There are at least four fault-free edges between any two subgraphs, and thus $C$ is a complete graph. According to Claim \ref{Kn-(n-4)}, $C$ is paired 2-disjoint path coverable. Therefore, there exists a paired 2-DPC connecting $\{c_{s_1}, c_{s_2}\}$ and $\{c_{t_1},c_{t_2}\}$ in $C$.

When $n \geq 10$, $|F_k|\leq\left\lceil \frac{n^k - 1}{2} \right\rceil (n-2) - \frac{3n}{2}$. By Claim \ref{Kn-(n-4)}, if \( C \) is not a paired 2-DPC, then in \( BC_{n,k} \), there are at least \( n-3 \) pairs of subgraphs such that the number of faulty edges between them is greater than or equal to \( n^k - 1 \). Since $(n-3)(n^k-1)-F_k\geq(n-3)(n^k-1)-\frac{n^k}{2}(n-2)+\frac{3n}{2}=n^k(\frac{n-4}{2})+\frac{n}{2}+3>0$, the number of faulty edges in $C $ is less than $n-3$. Thus, $C$ is paired 2-disjoint path coverable. Therefore, we can construct a paired 2-DPC $P_1$ from $c_{s_1}$ to $c_{t_1}$ and $P_2$ from $c_{s_2}$ to $c_{t_2}$ in $C-F$, as shown in the Figure \ref{fig:Case4}(a).

In $BC[l_{s_1}]$, $|F\cap E(l_{s_1},l_{11})|<n^k-1$. Thus there exits a node $x_{11} \in V(BC[l_{s_1}]) \setminus s_1$ such that $(x_{11},n^{l_{11}}(x_{11}))\notin F$. According to Corollary \ref{H-path}, there exits Hamiltonian path $H$ from $s_1$ to $x_{11}$ in $BC[l_{s_1}]$. In $BC[l_{11}]$, $|F\cap E(l_{s_1},l_{11}))|<n^k-1$, Thus there exits $x_{12} \in V(BC[l_{11}]) \setminus n^{l_{11}}(x_{11})$ such that $(x_{12},n^{l_{12}}(x_{12}))\notin F$.  According to Corollary \ref{H-path}, there exits a Hamiltonian path $H_2$ from $n^{l_{11}}(x_{11})$ to $x_{12}$ in $BC[l_{11}]$.

Following this construction recursively, we can establish a path $H(s_1,t_1)$ from $s_1$ to $t_1$ that covers all nodes in the subgraphs belonging to $P_1$. Similarly, we can construct a path $H(s_2,t_2)$ from $s_2$ to $t_2$ that covers all nodes in the subgraphs belonging to $P_2$, as shown in the Figure \ref{fig:Case4}(b).
\end{proof}

\section{Summary}
Following the proposal in \cite{0Link},
this paper focuses on constructing the fault-tolerant paired 2-disjoint path covers (2-DPC) in BCube under the Partitioned Edge Fault (PEF) model, which is a novel fault model by  constraining the number of faulty edges in each dimension. 
We prove that BCube always contains a paired 2-DPC between any two distinct source–destination pairs, with the number of faulty edges at most 
$\sum_{i=1}^k(\max\{0, n^i-5\})+n-5$ for $n=7$, 
$\sum_{i=1}^k(\max\{0, n^i-5\})+n-4$ for $n\in \{4,5,6,8,9\}$, 
and $\sum_{i=1}^k(\left\lceil \frac{n^i - 1}{2} \right\rceil \cdot(n-2) - \frac{3}{2}n)+n-4$ for $n\geq10$.
This result extends the applicability of the PEF model, and offers theoretical foundations for enhancing the fault tolerance of BCube under large-scale edge failure scenarios.

This paper investigates the fault tolerance of Bcube under link failures. Future work will explore fault tolerance under hybrid failure scenarios, including switch, server, and link failures.

\ack{This work was supported by National Key Research and Development Program of China [2022YFA1006400 to Q. Cai].}

\nocite{*}

\bibliographystyle{unsrt}
\bibliography{DPC_BCUBE}

\end{document}